\documentclass[envcountsame,oribibl,letterpaper]{llncs}
\usepackage[utf8x]{inputenc}
\usepackage{amssymb}

\usepackage{amsmath}
\usepackage{hyperref}
\usepackage{url}
\usepackage{subfigure,graphicx,color}

 \let\doendproof\endproof  
\renewcommand\endproof{~\hfill\qed\doendproof}

\spnewtheorem{quest}[theorem]{Question}{\itshape}{\normalfont}
\spnewtheorem{obs}[theorem]{Observation}{\itshape}{\normalfont}
\spnewtheorem{conj}[theorem]{Conjecture}{\itshape}{\normalfont}
\newcommand{\conv}[1]{{\rm conv}(#1)}
\newcommand{\h}[1]{{\rm hn}(#1)}
\newcommand{\setm}{\backslash\hspace{-.05cm}}

\definecolor{cmar}{RGB}{39,118,125}

\newcommand{\hn}[1]{\ensuremath{hn(#1)}}

\title{Convexity in partial cubes: the hull number}
\author{Marie Albenque\inst{1} \and Kolja Knauer\inst{2}}
\institute{LIX UMR 7161, \'Ecole Polytechnique, CNRS, France \and LIF UMR 7279, Universit\'{e} Aix-Marseille, CNRS, France}

\begin{document}

\maketitle

\begin{abstract}
We prove that the combinatorial optimization problem of determining the hull number of a partial cube is NP-complete. This makes partial cubes the minimal graph class for which NP-completeness of this problem is known and improves earlier results in the literature.

On the other hand we provide a polynomial-time algorithm to determine the hull number of planar partial cube quadrangulations.

Instances of the hull number problem for partial cubes described include
poset dimension and hitting sets for interiors of curves in the plane.

To obtain the above results, we investigate convexity in partial cubes and obtain a new characterization of these graphs in terms of their lattice of convex subgraphs. This refines a theorem of Handa. Furthermore we provide a topological representation theorem for planar partial cubes, generalizing a result of Fukuda and Handa about tope graphs of rank $3$ oriented matroids. 
\end{abstract}

\section{Introduction}
The objective of this paper is the study of convexity and particularly of the hull number problem on different classes of partial cubes. Our contribution is twofold. First, we establish that the hull number problem is NP-complete for partial cubes, second, we emphasize reformulations of the hull number problem for certain classes of partial cubes leading to interesting problems in geometry, poset theory and plane topology. In particular, we provide a polynomial time algorithm for the class of planar partial cube quadrangulations.

Denote by $Q^d$ the hypercube graph of dimension $d$. A graph $G$ is called a \emph{partial cube} if there is an injective mapping $\phi:V(G)\to V(Q^d)$ such that $d_G(v,w)=d_{Q^d}(\phi(v),\phi(w))$ for all $v,w\in V(G)$, where, $d_G$ and $d_{Q^d}$ denote the graph distance in $G$ and $Q^d$, respectively. This is, for each pair of vertices of $\phi(G)$, at least one shortest path in $Q^d$ belongs to $\phi(G)$. In other words $\phi(G)$, seen as a subgraph of $Q^d$, is an \emph{isometric embedding} of $G$ in $Q^d$. One often does not distinguish between $G$ and $\phi(G)$ and just says that $G$ is an \emph{isometric subgraph} of $Q^d$.

Partial cubes were introduced by Graham and Pollak in~\cite{Gra-71} in the study of interconnection networks and continue to find strong applications; they form for instance the central graph class in media theory (see the recent book~\cite{Epp-08}) and frequently appear in chemical graph theory e.g.~\cite{Epp-09}. Furthermore, partial cubes ``present one of the central and most studied classes of graphs in all of the metric graph theory'', citing~\cite{Kla-12}. 

Partial cubes form a generalization of several important graph classes,  thus have also many applications in different fields of mathematics. 
This article discusses some examples of such families of graphs including Hasse diagrams of upper locally distributive lattices or equivalently antimatroids~\cite{Fel-09} (Section~\ref{sec:ULD}), region graphs of halfspaces and hyperplanes (Section~\ref{sec:NP}), and tope
graphs of oriented matroids~\cite{Cor-82} (Section~\ref{sec:plan}). 
These families contain many graphs defined on sets
of combinatorial objects: flip-graphs of strongly connected
and acyclic orientations of digraphs~\cite{Cor-07}, linear extension graphs of
posets~\cite{Tro-92} (Section~\ref{sec:linext}), integer tensions of digraphs~\cite{Fel-09}, configurations of chip-firing games~\cite{Fel-09}, to
name a few.

\medskip
Convexity for graphs is the natural counterpart of Euclidean convexity and is defined as follows; a subgraph $G'$ of $G$ is said to be \emph{convex} if all shortest paths in $G$ between vertices of $G'$ actually belong to $G'$. The \emph{convex hull} of a subset $V'$ of vertices -- denoted $\conv{V'}$ -- is defined as the smallest convex subgraph containing $V'$.
Since the intersection of convex subgraphs is clearly convex, the convex hull of $V'$ is the intersection of all the convex subgraphs that contain $V'$. 

A subset of vertices $V'$ of $G$ is a \emph{hull set} if and only if $\conv{V'}=G$. The \emph{hull number} or \emph{geodesic hull number} of $G$, denoted by $\hn{G}$, is the size of a smallest hull set. It was introduced in~\cite{Eve-85}, and since then has been the object of numerous papers. Most of the results on the hull number are about bounds for specific graph classes, see
e.g.~\cite{Cha-00,Her-05,Can-06,Cac-10,Dou-10,Cen-13}.
Only recently, in~\cite{Dou-09} the focus was set on computational aspects
of the hull number and it was proved that determining the hull number of a graph is
NP-complete. In particular, computing the convex hull of a given set of vertices was shown to be polynomial time solvable. The NP-completeness result was later strengthened to bipartite graphs in~\cite{Ara-13}. On the
other hand, polynomial-time algorithms have been obtained for unit-interval graphs,
cographs and split graphs~\cite{Dou-09}, cactus graphs and $P_4$-sparse
graphs~\cite{Ara-13}, distance hereditary graphs and chordal graphs~\cite{Kan-13}, and $P_5$-free triangle-free graphs in~\cite{Ara-13b}.
Moreover, in~\cite{Ara-13b}, a fixed parameter tractable algorithm to compute the hull number of any graph was obtained. Here, the parameter is the size of a vertex cover.

\medskip

Let us end this introduction with an overview of the results and the organization of this paper. 
Section~\ref{sec:cutpart} is devoted to properties of convexity in partial cubes and besides providing tools for the other sections, its purpose is to convince the reader that convex subgraphs of partial cubes behave nicely.
A characterization of partial cubes in terms of their convex subgraphs is given. In particular, convex subgraphs of partial cubes behave somewhat like polytopes in Euclidean space. Namely, they satisfy an analogue of the Representation Theorem of Polytopes~\cite{Zie-95}.

In Section~\ref{sec:NP} the problem of determining the hull number of a partial cube is proved to be NP-complete, improving earlier results of~\cite{Dou-09} and~\cite{Ara-13}. Our proof indeed implies an even stronger result namely that determining the hull number of a region graph of an arrangement of halfspaces and hyperplanes in Euclidean space is NP-complete.

In Section~\ref{sec:linext} the relation between the hull number
problem for linear extension graphs and the dimension problem of posets is discussed. We
present a quasi-polynomial-time algorithm to compute the dimension of a poset
given its linear extension graph and conjecture that the problem is polynomial-time solvable.
 
Section~\ref{sec:plan} is devoted to planar partial cubes. 
We provide a new characterization of this graph class, which is a topological representation theorem generalizing work of Fukuda and Handa on rank $3$ oriented matroids~\cite{Fuk-93}. This characterization is then used to obtain a polynomial-time algorithm that computes the hull number of planar partial cube quadrangulations. We conjecture the problem to be polynomial time solvable for general planar partial cubes.

In Section~\ref{sec:ULD} we study the lattice of convex subgraphs of a graph. First, we prove that given this lattice the hull-number of a graph can be determined in quasi-polynomial time and conjecture that the problem is indeed polynomial time solvable. We then prove that for any vertex $v$ in a partial cube $G$, the set of convex subgraphs of $G$ containing $v$ ordered by inclusion forms an upper locally distributive lattice. This leads to a new characterization of partial cubes, strengthening a theorem of Handa~\cite{Han-93}. 

We conclude the paper by giving the most interesting open questions in Section~\ref{sec:con}.

\section{Partial cubes and cut-partitions}\label{sec:cutpart}
All graphs studied in this article are connected, simple and undirected. 
Given a connected graph $G$ a \emph{cut} $C\subseteq E$ is a set of edges whose removal partitions $G$ into exactly two connected components. These components are called its \emph{sides} and are denoted by $C^+$ and $C^-$. For $V'\subset V$, a cut $C$ \emph{separates} $V'$ if both $C^+\cap V'$ and $C^-\cap V'$ are not empty. A \emph{cut-partition} of $G$ is a set $\mathcal{C}$ of cuts partitioning $E$. For a cut $C\in\mathcal{C}$ and $V'\subseteq V$ define $C(V')$ as $G$ if $C$ separates $V'$ and  as the side of $C$ containing $V'$, otherwise.

\begin{obs}\label{obs:bipartite}
 A graph $G$ is bipartite if and only if $G$ has a cut-partition.
\end{obs}


The equivalence classes of the \emph{Djokovi{\'c}-Winkler relation} of a partial cube~\cite{Djo-73,Win-84} can be interpreted as the cuts of a cut-partition. The following characterization of partial cubes is a reformulation of some properties of the Djokovi{\'c}-Winkler equivalence classes as well as some results from~\cite{Che-86,Ban-89}. We provide a simple self-contained proof.

\begin{theorem}~\label{thm:equiv}
 A connected graph $G$ is a partial cube if and only if $G$ admits a cut-partition $\mathcal{C}$ satisfying one of the following equivalent conditions: 
\begin{itemize}
       \item[(i)] there is a shortest path between any pair of vertices using no $C\in\mathcal{C}$ twice
       \item[(ii)] no shortest path in $G$ uses any $C\in\mathcal{C}$ twice
       \item[(iii)] for all $V'\subseteq V: \conv{V'}=\bigcap_{C\in\mathcal{C}} C(V')$
       \item[(iv)] for all $v,w\in V: \conv{\{v,w\}}=\bigcap_{C\in\mathcal{C}} C(\{v,w\})$
\end{itemize}
We call such a cut-partition a \emph{convex cut-partition}.
\end{theorem}
\begin{proof}
We start by proving that the existence of a cut-partition satisfying $(i)$ is equivalent to $G$ being a partial cube. Assume that $G$ is a partial cube. Let $d$ be the minimal integer, such that $G$ can be embedded into $Q^d$. Fix such an embedding and label 
each edge $\{u,v\}$ by the (unique) coordinate for which $u$ and $v$ differ. Let $C_i$ denote the set of edges labeled $i$, since $d$ is minimal, the deletion of $C_i$ disconnects $G$. To see that $C_i$ is a cut: 
let $\{x,y\}$ and $\{u,v\}$ be two edges of $C_{i}$ such that $x$ and $u$ (resp. $y$ and $v$) have the same $i$-th coordinate. Any shortest path in $Q^d$ between $x$ and $u$ avoids $C_i$ and at least one of them belongs to the embedding of $G$ in $Q^d$. Hence, there is a path in $G$ from $x$ to $u$ (and similarly from $y$ to $v$) that do not contain any edges of $C_{i}$. Therefore, there exists a path from $x$ to $y$ which contains only $\{u,v\}$ as an edge of $C_{i}$. Hence $C_i$ is inclusion-minimal. 


Reciprocally, let $\mathcal{C}=(C_i)_i$ be a cut-partition satisfying $(i)$, map every $v\in V$ to the $(0,1)$-vector 
$(x(v)_i)_{i\in|\mathcal{C}|}:=(|C^+_i\cap\{v\}|)_{C_i\in\mathcal{C}}$. For $u,v \in V$ and $C\in \mathcal{C}$, if $u$ and $v$ lie on the same side (resp. on opposites sides) of $C$, then a path between them contains an even (resp. odd) number of edges of $C$. By $(i)$, there exists a shortest path between them that contains exactly one edge of each cut that separates them. The embedding of the latter path by $x$ yields a shortest path in $Q^{|\mathcal{C}|}$ between $x(v)$ and $x(u)$. It then defines an isometric embedding of $G$ into a hypercube, hence $G$ is a partial cube. 
\medskip

We now prove that the four conditions are equivalent: 

\noindent $(ii)\Rightarrow (iii):$ By $(ii)$ every shortest path crosses any $C\in\mathcal{C}$ at most once. Thus, for every $C\in\mathcal{C}$ its sides $C^+$ and $C^-$ are convex subgraphs of $G$. Since the intersection of convex sets is convex, the set $\bigcap_{C\in\mathcal{C}} C(V')$ is convex and since for every subset $V'$ and cut $C$, $\conv{V'}\subset C(V')$, we obtain that $\conv{V'}\subset \bigcap_{C\in\mathcal{C}} C(V')$.\\
Reciprocally, assume there exists $V'\subset G$ such that $\bigcap_{C\in\mathcal{C}} C(V')\setm\ \conv{V'}$ is not empty and pick an element $v$ in this set adjacent to some $w\in\conv{V'}$. Say $\{v,w\}\in C_i\in\mathcal{C}$, both $v$ and $w$ belong to $C_i(V')$, the latter must be equal to $G$ and hence $C_i$ separates $V'$. Let $\{x,y\}\in C_i$ with $x,y\in \conv{V'}$. By Observation~\ref{obs:bipartite}, $G$ is bipartite and we can assume for some $k$ that $d_G(x,w)=k$ and $d_G(y,w)=k+1$. By $(ii)$ no shortest $(x,w)$-path $P$ may use an edge of $C_i$, because otherwise a shortest $(y,w)$-path would use two edges of $C_i$. Extending $P$ to a $(y,v)$-path $P'$ of length $k+2$ cannot yield a shortest path because $P'$ uses $C_i$ twice. Thus, $d_G(y,v)\leq k+1$ but by bipartiteness we have $d_G(y,v)=k$. So there is a shortest $(y,v)$-path $P$ of length $k$ which therefore does not use $w$. Extending $P$ to $w$ yields a shortest $(y,w)$-path using $v$. Hence $v\in\conv{V'}$.

\noindent $(iii)\Rightarrow (iv)$ is clear. To prove $(iv)\Rightarrow (i)$, if there exists a shortest path between $v,w \in V$ that uses a $C_i$ more than once, then there exist two vertices $x$ and $y$ along this path, so that a shortest path between $x$ and $y$ uses $C_i$ exactly twice. This contradicts $(iv)$ with respect to $\conv{x,y}$.

\noindent $(i)\Rightarrow (ii):$ note that if $C_i$ separates $x$ and $y$, then all paths between $x$ and $y$ must contain at least one edge of $C_i$. Hence, if there is one shortest path using each of those cuts exactly once, then any shortest path must also use exactly once this set of cuts. 
\end{proof}

Note that $(iii)$ resembles the
\emph{Representation Theorem for Polytopes}, see~\cite{Zie-95}; where the
role of points is taken by vertices and the halfspaces are mimicked by the sides
of the cuts in the cut-partition. 

Note that $(iii)$ gives an easy polynomial time algorithm to compute the convex hull of a set $V'$ of vertices, by just taking the intersection of all sides containing $V'$. But as mentioned earlier, this is polynomial time also for general graphs~\cite{Dou-09}. More importantly, thanks to $(iii)$, the hull number problem has now a very useful interpretation as a hitting set problem:

\begin{corollary}\label{cor:hull}
Let $\mathcal{C}$ be a convex cut-partition then $V'$ is a hull set if and only if there exists a vertex of $V'$ on both sides of $C$, for all $C\in\mathcal{C}$.
\end{corollary}

 With a little more work one gets:
\begin{corollary}\label{cor:onesidedhull}
Let $\mathcal{C}$ be a convex cut-partition. For $v\in V$ denote by $h_v$ the size of a smallest set of vertices hitting $V\setminus C(v)$ for all $C\in \mathcal{C}$. We have $\hn{G}=\min_{v\in V}h_v+1$.
\end{corollary}
\begin{proof}
 Extending a hitting set of size $h_u$ by $u$ yields a hitting set of all sides of $\mathcal{C}$ of size at most $h_u+1$. Thus, $\hn{G}\leq\min_{v\in V}h_v+1$, by Corollary~\ref{cor:hull}. 
 
Conversely let $H$ be a minimal hull set, i.e. such that $|H| = \hn{G}$. By Corollary~\ref{cor:hull}, $H$ has a vertex on both sides of $C$, for any $C$ in the cut partition. Hence, for $u$ in $H$, $H\backslash\{u\}$ hits $V\setminus C(u)$ for all $C\in \mathcal{C}$. Therefore, the size of $H\backslash \{u\}$ is at most $h_u$. It leads to $\hn{G}-1\geq h_u\geq \min_v h_v$, which concludes the proof. 
\end{proof}

\section{NP-completeness of hull number in partial cubes}\label{sec:NP}
The section is devoted to the proof of the following result: 

\begin{theorem}\label{thm:NP}
 Given a partial cube $G$ and an integer $k$ it is NP-complete to decide whether $\h{G}\leq k$.
\end{theorem}
\begin{proof}
Computing the convex hull
of a set of vertices is doable in polynomial-time in general graphs, see e.g.~\cite{Dou-09}, i.e., the problem is in NP. To prove
the NP-completeness, we exhibit a reduction from the problem SAT-AM3 described below and known to
be NP-complete~\cite{Gar-79}. Let us first recall some logical terminology. Any Boolean variable $x$ can produce two literals: either non-negated and denoted $x$ (with a slight abuse of notation) or negated and denoted $\bar x$.  
 
\begin{quote}
\begin{description}
\item[SAT-AM3:]
\item[Instance:] A formula $F$ in Conjunctive Normal Form on $m$ clauses $D_1,\ldots, D_m$, each consisting of at most three literals on variables $x_1,\ldots, x_n$. Each variable \emph{appears in at most three clauses}. \\
In other words, we can write $F=D_1\wedge D_2\wedge \dots \wedge D_m$, where for all $i \in \{1,\ldots,m\}$, $D_i = x\vee y \vee z$ or $D_i=x\vee y$, with $x,y,z \in \{x_1,\bar x_1,x_2,\bar x_2,\ldots\}$. For each variable $x$,
 \[
 \#\{i, \text{ such that } x\in D_i\text{ or }\bar x\in D_i\}\leq 3.
 \] 
\item[Question:] Is $F$ satisfiable?
\end{description}
\end{quote}

Given an instance $F$ of SAT-AM3, we construct a partial cube $G_F$ such that $F$ is satisfiable if and only if $\h{G_F}\leq n+1$. 

\medskip
Given $F$ we start with two vertices $u$ and $u'$ connected by an edge. For each $1\leq i \leq m$, introduce a vertex $d_i$ and link it to $u$. If two clauses, say $D_i$ and $D_j$, share a literal, add a new vertex $d_{i,j}$ and connect it to both $d_i$ and $d_j$.

Now for each variable $x$, introduce a copy $G_x$ of the subgraph induced by $u$ and the vertices corresponding to clauses that contain $x$ (including vertices of the form $d_{\{i,j\}}$ in case $x$ appears as the same literal in $D_i$ and $D_j$). Assume without loss of generality that each Boolean variable $x$ used in $F$ appears at least once non-negated and once negated. Then, each literal appears at most twice in $F$ and the two possible options for $G_x$ are displayed on Figure~\ref{fig:gadget}. Label the vertices of $G_x$ as follows. The copy of $u$ is labeled $u_x$. If the literal $x$ (resp. $\bar x$) appears only once in $F$ -- say in $D_i$ -- label $v_x$ (resp. $\bar v_x$) the copy of $d_i$ in $G_x$ (see Figure~\ref{fig:gadget2}). If it appears twice -- say in $D_i$ and $D_j$ -- label respectively $d_{i,x}$ and $d_{j,x}$ the copies of $d_i$ and $d_j$ and label $v_x$ (resp. $\bar v_x$) the copy of $d_{i,j}$ (see Figure~\ref{fig:gadget2}). Connect $G_x$ to the rest of the graph by introducing a matching $M_x$ connecting each original vertex with its copy in $G_x$ and call $G_F$ the graph obtained.
\medskip
\begin{figure}[ht]
 \centering\subfigure[\label{fig:gadget2}The variable $y$ appears twice in $F$,]{\includegraphics[width=.45\linewidth,page=2]{NPhull.pdf}}\qquad \subfigure[\label{fig:gadget3} the variable $x$ three times.]{\includegraphics[width=.45\linewidth,page=1]{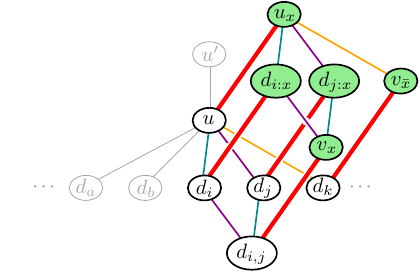}}
\caption{General structure of the graph $G_F$, with the two possible examples of gadgets associated to a variable. Red edges correspond to the cut $M_y$ on~\subref{fig:gadget2} and $M_x$ on~\subref{fig:gadget3}.} 
\label{fig:gadget}
\end{figure}

Observe first that $G_F$ is a partial cube. Define a cut partition of $G_F$ into $n+m+1$ cuts as follows. One cut consists of the edge $(u,u')$. The cut associated to a clause $D_i$ contains the edge $\{u,d_i\}$, any edge of the form $\{d_{\{i,j\}},d_j\}$ and all the copies of such edges that belong to one of the $G_x$. Let us call this cut $C_i$. Finally, the cut associated to a variable $x$ is equal to $M_x$. This cut partition satisfies the property $(i)$ of Theorem~\ref{thm:equiv}. Indeed, for a cut $C$ denote respectively $\partial^+C$ and $\partial^-C$ the vertices in $C^+$ and $C^-$ incident to edges of $C$. Property~$(i)$ of Theorem~\ref{thm:equiv} is in fact equivalent to say that, for each cut $C\in \mathcal{C}$, there exists a shortest path between any pair of vertices of $\partial^+C$ or $\partial^-C$, that contains no edge of $C$. A case by case analysis of the different cuts in $G_F$ concludes the proof. 
%
%
%
\medskip

Assume $F$ is satisfiable and let $S$ be a satisfying assignment of variables. Let $H$ be the union of $\{u'\}$ together with the following subset of vertices. For each variable $x$, $H$ contains the vertex $v_x$ if $x$ is set to true in $S$ or the vertex $\bar v_{x}$ otherwise. Let us prove that $H$ is a hull set. Since $u$ belongs to any path between $u'$ and any other vertex, $u$ belongs to $\conv H$. Moreover, for each variable $x$, the vertex $u_x$ lies on a shortest path both between $v_x$ and $u'$ and between $\bar v_{x}$ and $u'$, hence all the vertices $u_x$ belong to $\conv H$. Next, for each literal $\ell$ and for each clause $D_i$ that contains $\ell$, there exists a shortest path between $u'$ and $v_\ell$ that contains $d_i$. Then, since $S$ is a satisfying assignment of $F$, each clause vertex belongs to $\conv H$. It follows that $\conv H$ also contains all vertices $d_{i,j}$.
 
To conclude, it is now enough to prove that for all $\ell\notin S$, the vertex $v_\ell$ also belongs to $\conv H$. In the case where $\ell$ appears in only one clause $D_i$, then $v_\ell$ belongs to a shortest path between $d_i$ and $u_\ell$. In the other case, $v_\ell$ belongs to a shortest $(u_\ell,d_{i,j})$-path. Thus, $\conv H =G$. 

%
\medskip

Assume now that there exists a hull set $H$, with $|H|\leq n+1$. By Corollary~\ref{cor:hull}, the set $H$ necessarily contains $u'$ and at least one vertex of $G_x$ for each variable $x$. This implies that $|H|=n+1$ and therefore for all variables $x$, $H$ contains exactly one vertex $w_x$ in $G_x$. Since any vertex of $G_x$ lies either on a shortest $(u',v_x)$-path or $(u',\bar v_{x})$-path, we can assume that $w_x$ is either equal to $v_x$ or to $\bar v_{x}$.  Hence, $H$ defines a truth assignment $S$ for $F$. Now let $C_i$ be the cut associated to the clause $D_i$ and let $C_i^+$ be the side of $C_i$ that contains $d_i$. Observe that if $v_x$ belongs to $C_i^+$, then $x$ appears in $D_i$. By Corollary~\ref{cor:hull}, $H$ intersects $C_i^+$, hence there exists a literal $\ell$ such that $v_\ell$ belongs to $H$. Thus, $H$ encodes a satisfying truth-assignment of $F$.

%
\end{proof}

The gadget in the proof of Theorem~\ref{thm:NP} is a relatively special partial cube and the statement can thus be strengthened. For a polyhedron $P$ and a set $\mathcal{H}$ of hyperplanes in $\mathbb{R}^d$, the \emph{region graph} of $P\setminus\mathcal{H}$ is the graph whose vertices are the connected components of $P\setminus\mathcal{H}$ and where two vertices are joined by an edge if their respective components are separated by exactly one hyperplane of $\mathcal{H}$. Note that if we set $P=\mathbb{R}^d$, then the region graph is just the usual region graph of the hyperplane arrangement $\mathcal{H}$ and therefore has hull number $2$. However, adding $P$ to the setting increases the difficulty. Indeed, the proof of Theorem~\ref{thm:NP} can be adapted 
to obtain:
\begin{corollary}\label{cor:NP}
 Let $P\subset \mathbb{R}^d$ be a polyhedron and $\mathcal{H}$ a set of hyperplanes. It is NP-complete to compute the hull number of the region graph of $P\setminus\mathcal{H}$.
\end{corollary}
\begin{proof}
We show how to represent the gadget $G_F$ from the proof of Theorem~\ref{thm:NP} as the region graph of $P\setminus\mathcal{H}$.
Define $P$ as those $y\in \mathbb{R}^{m+n+2}$ with

\begin{center}
\begin{tabular}{ll}
  $y_i\geq 0$, & $i\in \{0,\ldots, m+n+1\}$\\
  $y_0+y_i\leq 2$, & $i\in \{1,\ldots, m+n+1\}$\\
  $y_i+y_j\leq 2$, & $\{i,j\}\subseteq \{1,\ldots, m\mid i\neq j; D_i\cap D_j=\emptyset\}$\\
  $y_i+y_j+y_k\leq 3$, & $\{i,j\}\subseteq\{1,\ldots, m\mid i\neq j; D_i\cap D_j\neq\emptyset\},$\\ ~ & $k\in \{m+1,\ldots, m+n+1\mid x_{k-m},\overline{x_{k-m}}\notin D_i\cap D_j\}$\\
  $y_i+y_j\leq 2$, & $i\in \{1,\ldots, m\},$\\ 
  ~ & $j\in \{m+1,\ldots, m+n+1\mid x_{k-m},\overline{x_{k-m}}\notin D_i\}$\\
  $y_{m+1}+\ldots +y_{m+n+2}\leq 2$ & ~
\end{tabular}
\end{center}

We slice $P$ with the family $\mathcal{H}$ of hyperplanes consisting of $H_i:=\{y\in \mathbb{R}^{m+n+2}\mid y_i=1\}$ for all $i\in \{0,\ldots, m+n+1\}$. Every vertex $v$ of $G_F$ will be identified with a cell $R_v$ of $\mathcal{H}\cap P$. In the following we describe the cells associated to vertices by giving strict inequalities for the $y\in \mathbb{R}^{m+n+2}$ in their interior. 

\begin{center}

\begin{tabular}{lcl}
  $u$ & $\rightsquigarrow$ & $0<y_k<1 \text{ for all }k$\\
  $u'$ & $\rightsquigarrow$ & $0<y_k<1 \text{ for }k\geq 1 ; 1<y_0<2$\\
  $d_j$ & $\rightsquigarrow$ & $0<y_k<1 \text{ for }k\neq j ; 1<y_j<2$\\
  $d_{\{i,j\}}$ & $\rightsquigarrow$ & $0<y_k<1 \text{ for }k\neq i,j \text{ ; } 1<y_i,y_j<2; y_i+y_j<2$\\
  $u_{x_{\ell}}$ & $\rightsquigarrow$ & $0<y_k<1 \text{ for }k\neq m+{\ell} ; 1<y_{m+{\ell}}<2$\\
  $d_{j:x_{\ell}}$ & $\rightsquigarrow$ & $0<y_k<1 \text{ for }k\neq j, m+{\ell} ; 1<y_j,y_{m+{\ell}}<2$\\
  $d_{\{i,j\}:x_{\ell}}$ & $\rightsquigarrow$ & $0<y_k<1 \text{ for }k\neq i,j,m+\ell \text{ ; } 1<y_i,y_j,y_{m+\ell}<2; y_i+y_j<2$\\
\end{tabular}

\end{center}

\end{proof}

\section{The hull number of a linear extension graph}\label{sec:linext}
%
Given a poset $(P,\leq_P)$, a \emph{linear extension} $L$ of $P$ is a total order $\leq_L$ on the elements
of $P$ compatible with $\leq_P$, i.e., $x\leq_P y$ implies $x\leq_L y$. 
The set of vertices of the \emph{linear extension graph} $G_L(P)$ of $P$ is the set of all linear extensions of $P$ and there is an edge between $L$ and $L'$ if and only if $L$ and
$L'$ differ by a neighboring transposition, i.e., by reversing the order of two consecutive elements. 

Let us see that property $(i)$ of Theorem~\ref{thm:equiv} holds for $G_L(P)$. Each incomparable pair $x\parallel y$ of $(P,\leq_P)$ corresponds to a cut of $G_L(P)$ consisting of the edges where $x$ and $y$ are reversed. The set of these cuts is clearly a cut-partition of $G_L(P)$. Observe then that the distance between two linear extensions $L$ and $L'$ in $G_L(P)$ is equal to the number of pairs that are ordered differently in $L$ and $L'$, i.e., no pair $x\parallel y$ is reversed twice on a shortest path. Hence $G_L(P)$ is a partial cube. 

A \emph{realizer} of a poset is a set $S$ of linear extensions such that their intersection is $P$. In other words, for every incomparable pair $x\parallel y$ in $P$, there exist $L,L'\in S$ such that $x<_L y$ and $x>_{L'} y$. It is equivalent to say that, for each cut $C$ of the cut-partition of $G_L(P)$, the sets $C^+\cap S$ and $C^-\cap S$ are not empty. By Corollary~\ref{cor:hull}, it yields a one-to-one correspondence between realizers of $P$ and hull sets of $G_L(P)$. In particular the size of a minimum realizer -- called the \emph{dimension} of the poset and denoted $\dim(P)$ -- is equal to the hull number of $G_L(P)$.
The dimension is a fundamental parameter in poset combinatorics, see e.g.~\cite{Tro-92}. In particular, for every \emph{fixed} $k\geq 3$, it is NP-complete to decide if a given poset has dimension at least $k$, see~\cite{Yan-82}. But if instead of the poset its linear extension graph is considered to be the input of the problem, then we get:

\begin{proposition}\label{prop:lin}
 The hull number of a linear extension graph (of size $n$) can be determined in time $O(n^{\log n + 3})$, i.e., the dimension of a poset $P$ can be computed in quasi-polynomial-time in $G_L(P)$.
\end{proposition}
\begin{proof}
An antichain in a poset is a set of mutually incomparable elements of $P$ and the width $\omega(P)$ of $P$ is the size of the largest antichain of $P$, see~\cite{Tro-92}. It is a classic result that $\dim(P)\leq \omega(P)$. Since any permutation of an antichain appears in at least one linear extension, $\omega(P)!\leq n$ and therefore $\dim(P)\leq \log(n)$. Thus, to determine the hull-number of $G_L(P)$ it suffices to compute the convex hull of all subsets of at most $\log(n)$ vertices. Since the convex hull can be computed in cubic time, by~\cite{Dou-09}, we get the claimed upper bound.
\end{proof}

Note that the linear extension graph of a poset is the region graph of a hyperplane arrangement and a polyhedron. Indeed, the \emph{order-polytope} of a poset $P$ on $d$ elements is defined by $\{x\in\mathbb{R}^d\mid i\leq_P j\Rightarrow 0\leq x_i\leq x_j\leq 1\}$. The \emph{braid-arrangement} in dimension $d$ consists of the hyperplanes $H_{i,j}:=\{x\in\mathbb{R}^d\mid x_i=x_j\}$. Stanley shows that the linear extension graph of a poset $P$ arises as the
region graph of the {order-polytope} of $P$ and the {braid-arrangement}, see~\cite{Sta-86}. Hence, determining the dimension
of a poset given its linear extension graph is a special case of the problem of
Corollary~\ref{cor:NP}. However, having Proposition~\ref{prop:lin} since it is widely believed that problems solvable in quasi-polynomial are not NP-complete (this follows from the Exponential Time Hypothesis~\cite{Imp-01}), we conjecture:

\begin{conjecture}\label{conj:linext}
 The dimension of a poset given its linear extension graph can be determined in polynomial-time.
\end{conjecture}

\section{Planar partial cubes}\label{sec:plan}

Surprisingly enough, some results about oriented matroids and more specifically about the tope graphs of oriented matroids can be rephrased in the context of partial cubes. We refer the interested reader to~\cite{Bjo-99} for a thorough introduction to oriented matroids and their many applications. 
We introduce here only the definitions needed to state our results.

A \emph{Jordan curve} is a simple closed curve in the plane. For a Jordan curve $S$ its complement $\mathbb{R}^2\setminus S$ has two components: one is bounded and is called its \emph{interior}, the other one, unbounded, is called its \emph{exterior}. The closure of the interior and of the exterior of $S$ are denoted $S^+$ and $S^-$, respectively. An \emph{arrangement} $\mathcal{S}$ of Jordan curves is a set of of Jordan curves such that if two of them intersect they do so in a finite number of points. 
The \emph{region graph} of an arrangement $\mathcal{S}$ of Jordan curves is the graph whose vertices are the connected components of $\mathbb{R}^2\setminus\mathcal{S}$, where two vertices are neighbors if their corresponding components are separated by exactly one element of $\mathcal{S}$. 

An \emph{antipodal partial cube} $G$ is a partial cube such that for every $u\in G$ there is a $\overline{u}\in G$ with $\conv{u,\overline{u}}=G$. In particular we have $\hn{G}=2$. The characterization of tope graphs of oriented matroids of rank at most $3$ by Fukuda and Handa 
may be rephrased as: 

\begin{theorem}[\cite{Fuk-93}]\label{thm:fukhanoriginal}
 A graph $G$ is an antipodal planar partial cube if and only if $G$ is the region graph of an arrangement $\mathcal{S}$ of Jordan curves such that for every $S, S'\in \mathcal{S}$ we have $|S\cap S'|=2$ and for $S, S', S''\in \mathcal{S}$ either $|S\cap S'\cap S''|=2$ or $|S^+\cap S'\cap S''|=|S^-\cap S'\cap S''|=1$.
\end{theorem}

Given a Jordan curve $S$ and a point $p\in \mathbb{R}^2\setminus S$ denote by $S(p)$ the closure of the side of $S$ not containing $p$.
 An arrangement $\mathcal{S}$ of Jordan curves is called \emph{non-separating}, if for any $p\in \mathbb{R}^2\setminus\mathcal{S}$ and subset $\mathcal{S}'\subseteq \mathcal{S}$ the set $\mathbb{R}^2\setminus\bigcup_{S\in\mathcal{S}'}S(p)$ is connected. 
 
 Two important properties of non-separating arrangements are summarized in the following:
 \begin{obs}\label{obs:pseudo}
  Let $\mathcal{S}$ be a non-separating arrangement. Then the closed interiors of $\mathcal{S}$ form a family of \emph{pseudo-discs}, i.e., different curves $S,S'\in \mathcal{S}$ are either disjoint, touch in exactly one point or cross in exactly two points.
  Moreover, the closed interiors of curves in $\mathcal{S}$ have the \emph{topological Helly property}: if three of them mutually intersect, then all three of them have a point in common. 

In Figure~\ref{fig:pseudomore} and Figure~\ref{fig:pseudocomm}, we illustrate how violating respectively the pseudodisc or the topological Helly property violates the property of being non-separating.
 \end{obs}

\begin{figure}[t]
\centering
\subfigure[\label{fig:pseudomore}Two curves intersecting in more than $2$ points.]{\includegraphics[width=.45\linewidth,page=2]{pseudo.pdf}}\qquad \subfigure[\label{fig:pseudocomm}Three interiors of curves mutually intersecting without a common point.]{\includegraphics[width=.45\linewidth,page=1]{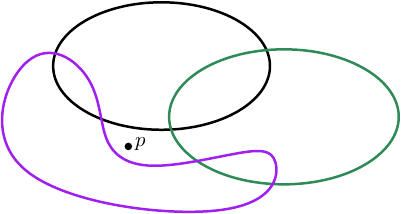}}
\caption{\label{fig:pseudo}Illustration of Observation~\ref{obs:pseudo}. In both cases $p$ is a point proving that the arrangement is not non-separating.}
\end{figure}
 
Non-separating arrangements of Jordan curves yield a generalization of Theorem~\ref{thm:fukhanoriginal}. 

\begin{theorem}\label{thm:fukhan}
 A graph $G$ is a planar partial cube if and only if $G$ is the region graph of
a non-separating arrangement $\mathcal{S}$ of Jordan curves.
\end{theorem}
\begin{proof}
This proof is illustrated in Figure~\ref{fig:planarcube}.
 Let $G$ be a planar partial cube with cut-partition $\mathcal{C}$. We consider
$G$ with a fixed embedding and denote by $G^*$ the planar dual. By planar
duality each cut $C\in\mathcal{C}$ yields a simple cycle $S_C$ in $G^*$. The set
of these cycles, seen as Jordan curves defines $\mathcal{S}$. Since $(G^*)^*=G$
the region graph of $\mathcal{S}$ is isomorphic to $G$. 

Now let $p\in \mathbb{R}^2\setminus \mathcal{S}$, which corresponds to vertex $v$ of $G$. Any curve $S$ in $\mathcal{S}$ corresponds to a cut $C$ and the region $S(p)$ contains exactly the subgraph $V\setminus C(v)$ of $G$. By Theorem~\ref{thm:equiv}(ii) we have that $V\setminus C(v)$ is a convex subgraph. Hence, for any $\mathcal{S}'\subseteq\mathcal{S}$, the region $\mathbb{R}^2\setminus\bigcup_{S\in\mathcal{S}'}S(p)$ contains exactly the intersection of all $C(v)$, for $C$ being associated to an element of $\mathcal{S}'$. Since the intersection of convex subgraphs is convex, this graph is convex and in particular therefore connected. Hence, $\mathbb{R}^2\setminus\bigcup_{S\in\mathcal{S}'}S(p)$ is connected and $\mathcal{S}$ is non-separating.

%
 
 Conversely, let $\mathcal{S}$ be a set of
Jordan curves and suppose its region graph $G$ is not a partial cube. In
particular the cut-partition $\mathcal{C}$ of $G$ arising by
dualizing $\mathcal{S}$ does not satisfy Theorem~\ref{thm:equiv} $(i)$. This
means there are two regions $R$ and $T$ such that every curve $S$ contributing to the
boundary of $R$ contains $R$ and $T$ on the same side, i.e., for any $p\in R\cup T$ and such $S$ we have $R, T\subseteq S(p)$. Let
$\mathcal{S}'$ be the union of these curves. The union
$\bigcup_{S\in\mathcal{S}'}S(p)$ separates $R$ and $T$, i.e.,
$\mathbb{R}^2\setminus\bigcup_{S\in\mathcal{S}'}S(p)$ is not connected.
\end{proof}
\begin{figure}[ht]
\centering
 \includegraphics[width=0.8\textwidth]{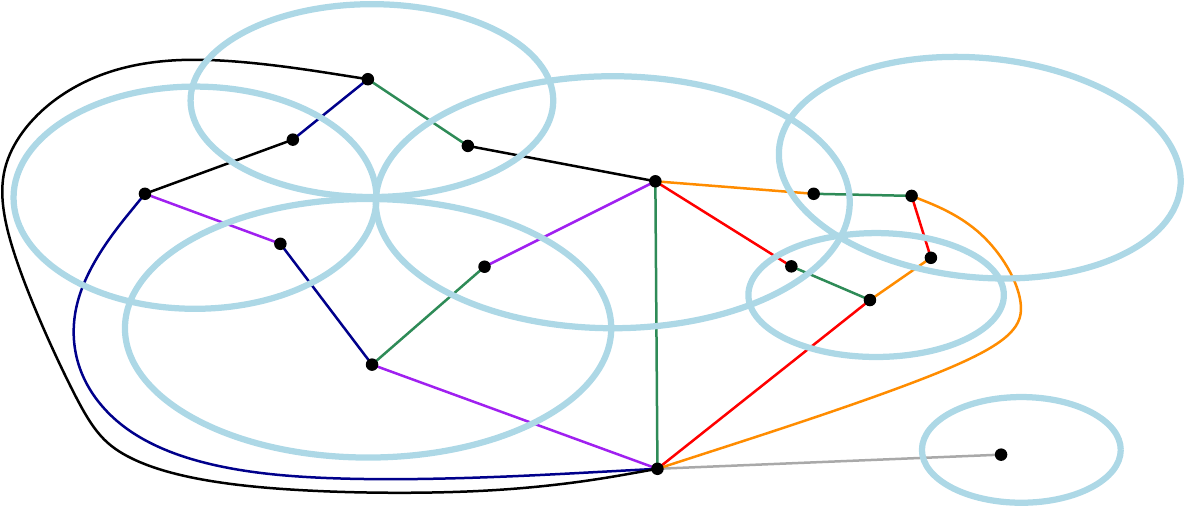}
\caption{\label{fig:planarcube}A (non-simple) non-separating set of Jordan curves and its region graph.}
\end{figure}

A set of Jordan curves is \emph{simple} if no point of the plane is contained in more than two curves. 

\begin{lemma}\label{lem:hitint}
 A minimum hitting set for (open) interiors of a non-separating simple set $\mathcal{S}$ of Jordan curves can be computed in polynomial-time.
\end{lemma}
\begin{proof}
We start by proving that finding a minimum hitting set for open interiors of a non-separating simple set $\mathcal{S}$ of Jordan curves is equivalent to finding a minimum clique cover of the intersection graph $I(\mathcal{S})$ of open interiors of $\mathcal{S}$:

First, since all open interiors containing a given point in the plane form a clique in $I(\mathcal{S})$ every hitting set corresponds to a clique cover of $I(\mathcal{S})$.

Second, we show that every inclusion maximal clique in $I(\mathcal{S})$ corresponds to a set of curves that have a point contained in the intersection of all their interiors. This is clear for cliques of size $1$ or $2$. By Observation~\ref{obs:pseudo}, if the open interiors of every triple $S_1,S_2,S_3\in\mathcal{S}$ mutually intersect, then their \emph{closed} interiors have a common point. We can apply the planar version of the classical \emph{Topological Helly Theorem}~\cite{Hel-30}, i.e., for any family of \emph{closed homology cells} in which every triple
has a point in common, all members have a point in common. (Since a pseudodisc is homeomorphic to a disc, it has trivial homology, which is the definition of closed homology cell.) Thus, any maximal clique of $I(\mathcal{S})$ corresponds to a set of curves that have a point contained in the intersection of all their \emph{closed} interiors. Now, since $\mathcal{S}$ is simple the intersection of more than two mutually intersecting closed interiors actually has to contain a region and not only a point. Thus, the maximal cliques correspond to sets of curves that have a point contained in the intersection of all their open interiors.

Since in a minimum clique cover the cliques can be assumed to be maximal, finding a minimum clique cover of $I(\mathcal{S})$ is equivalent to finding a minimum hitting set for the open interiors of $\mathcal{S}$. 
\medskip

Now, we prove that $I(\mathcal{S})$ is chordal: assume that there is a chordless cycle $C_k$ whose vertices correspond to $S_1^+, \ldots, S_k^+$ for $k\geq 4$, while the arrangement $\mathcal{S}$ being simple.
To prove that $\mathcal{S}$ is separating, we distinguish two cases:

First, assume $C_k$ is also the intersection graph of the closures of $S_1^+, \ldots, S_k^+$. That is, the closures of $S_i^+$ and $S_j^+$ intersect if and only if $|i-j|=1$ and in this case, they intersect in two points. Starting with $S_1^+$ denote one of its intersection points with $S_2^+$ by $p_{12}$, the other one by $q_{12}$. From $p_{12}$ we can follow a segment on the boundary of $S_2^+$ until we reach the first intersection point $p_{13}$ with $S_3^+$ without passing $q_{12}$. Continuing like this we construct a closed curve on the boundary of the union of the closures of $S_1^+, \ldots, S_k^+$. Analogously, starting with $q_{12}$ we can construct another such curve disjoint to the first one. Thus, the boundary of the union of the closures of $S_1^+, \ldots, S_k^+$ has two components, i.e., the union is not simply connected. Therefore its removal disconnects the plane. Note that the constructed scenario resembles the one depicted in Figure~\ref{fig:pseudocomm}.

Second, suppose that there is a point $x_{ij}$ in the intersection of the boundaries of $S_i^+$ and $S_j^+$ with $|i-j|>1$, i.e., $x_{ij}$ is a \emph{touching point} of $S_i^+$ and $S_j^+$. If $x_{ij}$ is contained in another interior $S_{\ell}^+$, then $S_{\ell}^+$ cannot entirely contain neither $S_i^+$ nor $S_j^+$, since both have a further neighbor in $C_k$ implying that $S_{\ell}^+$ intersects at least three interiors. Hence, both $S_i^+\setminus S_{\ell}^+$ and $S_j^+\setminus S_{\ell}^+$ are non-empty. Therefore, $S_{\ell}^+\setminus S_{i}^+\setminus S_{j}^+$ is disconnected and choosing a point $p\in S_{\ell}^+\setminus S_{i}^+\setminus S_{j}^+$ yields a contradiction to non-separability with respect to $S_{\ell}(p), S_{i}(p), S_{j}(p)$ . If $x_{ij}$ is contained in no other interior, then locally around $x_{ij}$ there are four regions which alternatingly are either contained or not contained in some interior. Any curve from one of the regions not contained in an interior to the other one has to intersect the boundary of the union of the closures of $S_1^+, \ldots, S_k^+$, because it has to cross closed interiors corresponding to vertices on the paths from $S_i^+$ to $S_{j}^+$ in $C_k$. Thus, $\mathcal{S}$ is separating.

%

Finally, chordal graphs form a subset of perfect graphs and hence by~\cite{Gro-84} a minimum clique cover can be computed in polynomial time. 
\end{proof}

Note that the above proof relies on simplicity in a two-fold way. Consider for example the non-simple arrangement in Figure~\ref{fig:planarcube}. The intersection graph of open interiors contains induced four-cycles and cliques in the intersection graphs of closed interiors cease to be cliques in the intersection graph of open interiors.

\begin{theorem}\label{thm:poly}
 A minimal hitting-set for open interiors and exteriors of a non-separating simple set $\mathcal{S}$ of Jordan curves can be computed in polynomial-time.
\end{theorem}
\begin{proof}
Viewing $\mathcal{S}$ now as embedded on the sphere, any choice of a region $v$ as the unbounded region yields a different arrangement $\mathcal{S}_v$ of Jordan curves in the plane. Denote the size of a minimum hitting set of the interiors of $\mathcal{S}_v$ by $h_v$. By Corollary~\ref{cor:onesidedhull} we know that a minimum hitting set of exteriors and interiors of $\mathcal{S}$ is of size $\min_{v\in V}h_v+1$. Since by Lemma~\ref{lem:hitint} every $h_v$ can be computed in polynomial-time and $|V|$ is linear in the size of the input, this concludes the proof.

%

\end{proof}

Combining Corollary~\ref{cor:hull} and Theorems~\ref{thm:fukhan} and~\ref{thm:poly}, we get: 
\begin{corollary}
 The hull number of a plane quadrangulation that is a partial cube can be determined in polynomial-time.
\end{corollary}
Notice that in~\cite{Fow-81}, it was shown that the hitting set problem restricted to open interiors of (simple) sets of unit squares in the plane remains NP-complete and that the gadget used in that proof is indeed not non-separating.

Combined with Theorem~\ref{thm:poly}, a proof of the following conjecture would give a polynomial-time algorithm for the hull number of planar partial cubes.

\begin{conjecture}\label{conj:planar}
A minimum hitting set for open interiors of a non-separating set of Jordan curves can be found in polynomial-time.
\end{conjecture}

\section{The lattice of convex subgraphs}~\label{sec:ULD}
In this section we study the lattice of convex subgraphs of a graph. First, we give a quasi-polynomial-time algorithm to determine the hull-number of a graph given its lattice of convex subgraphs. Second, we present another indication for how nice partial cubes behave with respect to convexity: generalizing a theorem of Handa~\cite{Han-93} we characterize partial cubes in terms of their lattice of convex subgraphs, see Figure~\ref{fig:ULD} for an illustration.

We begin by introducing some notions of lattice theory:
A partially ordered set or poset $\mathcal{L}=(X,\leq)$ is a \emph{lattice}, if for each pair of elements $x,y\in \mathcal{L}$ there exist both a unique largest element smaller than $x$  and $y$ called their \emph{meet} and denoted $x\wedge y$, and a unique smallest element larger than $x$ and $y$ called their \emph{join} and denoted $x\vee y$. Since both these operations are associative, we can define $\bigvee M:=x_1\vee\ldots\vee x_k$ and $\bigwedge M:=x_1\wedge\ldots\wedge x_k$ for $M=\{x_1,\ldots,x_k\}\subseteq\mathcal{L}$. In our setting $\mathcal{L}$ will always be finite, which implies that it has a unique global maximum $\mathbf{1}$ and a unique global minimum $\mathbf{0}$. This allows to furthermore define $\bigvee \emptyset:=\mathbf{0}$ and $\bigwedge \emptyset=:\mathbf{1}$ as respectively the minimal and maximal element of $\mathcal{L}$. 

For $\mathcal{L}=(X,\leq)$ and $x,y\in X$, one says that $y$ \emph{covers} $x$ and writes $x\prec y$ if and only if $x<y$ and there is no $z\in X$ such that $x<z<y$. The \emph{Hasse diagram} of $\mathcal{L}$ is then the directed graph on the elements of $X$ with an arc $(x,y)$ if $x\prec y$. The classical convention is to represent a Hasse diagram as an undirected graph but embedded in the plane in such a way that the orientation of edges can be recovered by orienting them in upward direction.

An element $a$ of $\mathcal{L}$ is called \emph{atom} if $\mathbf{0}\prec a$ and a lattice is called \emph{atomistic} if every element of $\mathcal{L}$ can be written as join of atoms. The following is easy to see:

\begin{obs}\label{obs:lattice}
Given a graph $G$ the inclusion order on its convex subgraphs is an atomistic lattice $\mathcal{L}_G$. The atoms of $\mathcal{L}_G$ correspond to the vertices of $G$.
\end{obs}

 Now, the hull number of $G$ is the size of a minimum set $H$ of atoms of $\mathcal{L}_G$ such that $\bigvee H$ is the maximum of $\mathcal{L}_G$. This suggests the following generalization of hull number to atomistic lattices. The hull number $\hn{\mathcal{L}}$ of an atomistic lattice $\mathcal{L}$ is the size of a minimum set $H$ of atoms of $\mathcal{L}$ such that $\bigvee H$ is the maximum of $\mathcal{L}$. 

\begin{proposition}
 The hull number of an atomistic lattice $\mathcal{L}$ with $n$ elements and $a$ atoms can be computed in $O(a^{c\log n})$ time. In particular, given $\mathcal{L}_G$ the hull number of $G$ with $k$ vertices can be computed in $O(k^{c\log n})$ time.
\end{proposition}
\begin{proof}
 Let $\mathcal{L}$ be an atomistic lattice $\mathcal{L}$ with maximum $\mathbf{1}$ and set of atoms $A(\mathcal{L})$. Let $H\subseteq A(\mathcal{L})$ be inclusion minimal such that $\bigvee H=\mathbf{1}$. Then by minimality of $H$ for any two distinct subsets $H', H''\subsetneq H$ we have $\bigvee H'\neq \bigvee H''$: otherwise, $\bigvee H'\vee \bigvee (H\setminus H'')=\mathbf{1}$ but $H'\cup (H\setminus H'')\subsetneq H$. Thus, $2^{|H|}\leq |\mathcal{L}|$ and in particular $\hn{\mathcal{L}}\leq \log|\mathcal{L}|$. Hence, to compute the hull-number of $\mathcal{L}$ it suffices to compute the joins of all subsets of $A(\mathcal{L})$ of size at most $\log|\mathcal{L}|$. Assuming that the join of a set $A'$ of atoms can be computed in polynomial time, this gives the result.
 
 Note that, computing the join of a set of $A'$ of atoms in $\mathcal{L}_G$ corresponds to computing the convex hull of the vertices $A'$, which can be done in time polynomial in the size of $G$, see e.g.~\cite{Dou-09}. This yields the statement about graphs.
\end{proof}

As with Conjecture~\ref{conj:linext} it seems natural to conjecture the following:
\begin{conjecture}\label{conj:lattice}
 Given an atomistic lattice $\mathcal{L}$ with maximum $\mathbf{1}$ and set of atoms $A(\mathcal{L})$. The minimum size of a subset $H\subseteq A(\mathcal{L})$ such that $\bigvee H=\mathbf{1}$ can be computed in polynomial time.
\end{conjecture}

\bigskip

We now proceed to the study of lattice of convex subgraphs of a partial cube. First, we need some further definitions:

An element $m$ of a lattice $\mathcal{L}$ is called \emph{meet-reducible} if it can be written as the meet of elements all different from itself and is called \emph{meet-irreducible} otherwise. It is easy to see that an element $x$ is meet-irreducible if and only if there is exactly one edge in the Hasse diagram leaving $x$ in upward direction. (Note that in particular the maximum of $\mathcal{L}$ is meet-reducible since it can be written as $\bigwedge \emptyset$.)

A lattice is called \emph{upper
locally distributive} or ULD if each of its
elements admits a unique minimal representation as the meet of meet-irreducible elements. In other words, for every $x\in\mathcal{L}$ there is a unique inclusion-minimal set $\{m_1,\ldots,m_k\}\subseteq\mathcal{L}$ of meet-irreducible elements such that $x=m_1\wedge\ldots\wedge m_k$.

ULDs were first defined by Dilworth~\cite{Dil-40} and have thereafter often
reappeared, see~\cite{Mon-85} for an overview until the mid 80s. In particular, the Hasse
diagram of a ULD is a partial cube, see e.g.~\cite{Fel-09}. 

Given a graph $G$ and a vertex $v$, denote
by $\mathcal{L}_G^v$ the lattice of convex subgraphs of $G$ containing $v$ endowed with inclusion order. A theorem of~\cite{Han-93} can then be rephrased as:
\begin{theorem}[\cite{Han-93}]
 Let $G$ be a partial cube and $v\in G$. Then $\mathcal{L}_G^v$ is a ULD.
\end{theorem}

\begin{figure}[t]
\centering
 \includegraphics[width=0.7\textwidth]{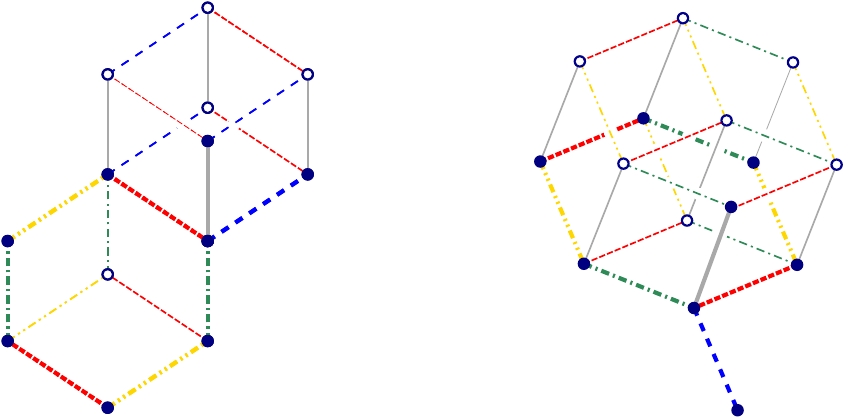}
\caption{Two ULDs obtained from the same partial cube (thick edges) by fixing a different vertex.} 
\label{fig:ULD}
\end{figure}

The following theorem extends Handa's theorem to a characterization of partial cubes according to the lattice of their convex subgraphs. 

\begin{theorem}\label{th:ULD}
Let $G$ be a graph, then the three following assertions are equivalent: 

$(i)$ $G$ is a partial cube.

$(ii)$ There exists a vertex $v$ in $G$, such that the poset $\mathcal{L}_G^v$ is a ULD whose Hasse diagram contains $G$ as an isometric subgraph. 

$(iii)$ For every vertex $v$, the poset $\mathcal{L}_G^v$ is a ULD whose Hasse diagram contains $G$ as an isometric subgraph. 
\end{theorem}

\begin{proof}
$(iii) \Rightarrow (ii)$ is immediate. We then start with $(ii)\Rightarrow (i)$. Since the diagram of any ULD is a partial cube, see e.g~\cite{Fel-09}, if $G$ is an isometric subgraph of it, then $G$ is a partial cube itself.

Now, to prove $(i) \Rightarrow (iii)$, let $G$ be a partial cube and $v\in V$. Since convexity is closed under  intersection, the set $\mathcal{L}_G^v$ equipped with the inclusion order is a lattice in which the meet-operation coincides with intersection. To characterize the meet-irreducible elements of $\mathcal{L}_G^v$, we rely on the following claim:

\begin{claim}
Given non-empty convex subgraphs $G',G''\subseteq G$, the subgraph $G''$ covers $G'$ in $\mathcal{L}_G^v$ if and only if there exists $u\in G''\backslash G'$ such that $G''=\conv{G',u}$ and such that there exists an edge $\{u,w\}$ in $G''$, where $w\in G'$.
 \end{claim}
\begin{proof}
If $G'\prec G''$, then there exists a set of vertices $U\subseteq G''\setminus G'$ such that $G''=\conv{G'\cup U}$. Consider one shortest path between a vertex of $G'$ and a vertex of $U$ and let $u$ be a vertex adjacent to $G'$ on this path. Then $G'\prec \conv{G',u}\leq G''$, and therefore $G''=\conv{G',u}$. 

Conversely, since $G'\subsetneq\conv{G',u}$, clearly $G'<G''$. Let now $C$ be the cut that contains $\{u,w\}$. Every shortest path between vertices in $G'\cup\{u\}$ uses only edges contained in $C$ or in cuts separating $G'$: if there is a shortest path from $u$ to $z\in G'$ using some $C'\neq C$, then since $C'$ is a cut some edge on any $(z,w)$-path is contained in $C'$.

Thus,  $\{C(V(G'))\mid C\in \mathcal{C}\}\setminus \{C(V(G''))\mid C\in \mathcal{C}\}=\{C(w)\}$. So by Theorem~\ref{thm:equiv}~(iii), there is no $G'''$ with $G'< G'''< G''$. 
 \end{proof}
A meet-irreducible element in a lattice is an element which is covered by exactly one other element. Therefore by the Claim the meet-irreducible elements  in $\mathcal{L}_G^v$ are precisely those convex subgraphs incident to precisely one $C\in\mathcal{C}$, i.e., the sides of the cuts in $\mathcal{C}$. 

Moreover every $G'$ has a unique minimal representation as intersection of sides -- take the sides of those cuts that contain edges emanating from $G'$. This is, in $\mathcal{L}_G^v$ every element has a unique minimal representation as meet of meet-irreducible elements, i.e., $\mathcal{L}_G^v$ is a ULD.

We are ready to show that the mapping $\phi(u):=\conv{v,u}$ is an isometric embedding of $G$ into $\mathcal{L}_G^v$:

If $\phi(u)=\phi(u')$, then by Theorem~\ref{thm:equiv} we have that $v$ and $u$ are separated by the same set of cuts as $v$ and $u'$. Since the cuts encode an embedding into the hypercube on a path from $u$ to $u'$ obtained by concatenating a shortest $(u,v)$-path with a shortest $(v,u')$-path each coordinate was changed twice or never, i.e., $u=u'$ and $\phi$ is injective. 

To see, that $\phi$ is edge-preserving let $\{u,u'\}$ be an edge. Then without loss of generality $u'$ is closer to $v$ than $u$. Thus by the Claim above we have $\conv{v,u'}\prec \conv{v,u}$. 

We still have to show that $\phi(G)$ is an isometric subgraph of the diagram of $\mathcal{L}_G^v$. 
From the proof that Hasse diagrams of ULDs are partial cubes it follows that the cuts in a cut-partition of a ULD correspond to its meet-irreducible elements, see~\cite{Fel-09}. Thus, in the cut partition of the diagram of $\mathcal{L}_G^v$, a cover relation $G'\prec G''$ is contained in a cut extending the cut $C\in\mathcal{C}$ of $G$ with $G'=G''\cap C(v)$. 
Now, a shortest path $P$ in $\phi(G)$ from $\conv{v,u}$ to $\conv{v,u'}$ corresponds to a path from $u$ to $u'$ in $G$ using no cut twice. Since cuts in $G$ correspond to cuts in $\mathcal{L}_G^v$, we have that $P$ does not use any cut of  $\mathcal{L}_G^v$ twice and is therefore also a shortest path of $\mathcal{L}_G^v$, by Theorem~\ref{thm:equiv}(ii).
\end{proof}


\section{Conclusions}\label{sec:con}
We studied the hull number problem for partial cubes and several subclasses. Apart from our contributions to the complexity of the hull number problem we think our most appealing results are the reinterpretations of this problem in seemingly unrelated mathematical settings. To illuminate this and reemphasize them, we repeat the conjectures made in this paper:

\begin{itemize}
 \item The dimension of a poset given its linear extension graph can be determined in polynomial-time. (Conjecture~\ref{conj:linext})
\item
A minimum hitting set for open interiors of a non-separating set of Jordan curves can be found in polynomial-time. (Conjecture~\ref{conj:planar})

\item
 Given an atomistic lattice $\mathcal{L}$ with maximum $\mathbf{1}$ and set of atoms $A(\mathcal{L})$. The minimum size of a subset $H\subseteq A(\mathcal{L})$ such that $\bigvee H=\mathbf{1}$ can be computed in polynomial-time. (Conjecture~\ref{conj:lattice})

\end{itemize}

\subsubsection{Acknowledgments.} 
The authors thank Victor Chepoi, Stefan Felsner, Matja{\v z} Kov{\v s}e, and Bartosz Walczak for fruitful
discussions. M.A.\ would also like to thank Stefan Felsner for his invitation to
the Discrete Maths group at the Technical University of Berlin, where this work
was initiated. M.A.\ acknowledges the
support of the ERC under the agreement ``ERC StG 208471 - ExploreMap" and of the ANR under the agreement ANR 12-JS02-001-01 ``Cartaplus''. K.K. was
supported by TEOMATRO (ANR-10-BLAN 0207) and DFG grant FE-340/8-1 as part of ESF
project GraDR EUROGIGA and PEPS grant EROS.

\bibliography{cublit}
\bibliographystyle{amsplain}
%
%
%
%
%
%

\end{document}